\documentclass[12pt,fleqn]{amsart}

\usepackage{amssymb}
\usepackage[utf8]{inputenc}
\usepackage{mathtools,amsmath,amsthm,amssymb}
\usepackage{enumerate}
\usepackage{color}
\newtheorem{thm}{Theorem}
\newtheorem{lemma}[thm]{Lemma}
\newtheorem{prop}[thm]{Proposition}
\newtheorem{mydef}[thm]{Definition}

\numberwithin{equation}{thm}
\numberwithin{thm}{section}

\newcommand{\fin}{\mbox{\it fin}}

\newcommand{\con}{\mathfrak c}
\newcommand{\be}{\mathfrak b}
\newcommand{\eps}{\varepsilon}
\newcommand{\alga}{\mathfrak A}
\newcommand{\algb}{\mathfrak B}

\newcommand{\cE}{{\mathcal E}}

\newcommand{\cJ}{{\mathcal J}}

\newcommand{\cS}{\protect{\mathcal S}}

\newcommand{\btu}{\bigtriangleup}

\newcommand{\vf}{\varphi}

\newcommand{\stevo}{Todor\v{c}evi\'c}

\newcommand{\sm}{\setminus}

\newcommand{\sub}{\subseteq}

\newcommand{\clop}{\protect{\rm Clop} }
\newcommand{\cov}{{\mbox{cov}}}
\newcommand{\cof}{{\mbox{cof}}}

\begin{document}
\title{Nonseparable growth of the integers\\
supporting a measure}

%    author two information
\author[P.\  Drygier]{Piotr Drygier}
\address{Instytut Matematyczny, Uniwersytet Wroc\l awski}
\email{piotr.drygier@math.uni.wroc.pl}

\author[G.\ Plebanek]{Grzegorz Plebanek}
\address{Instytut Matematyczny, Uniwersytet Wroc\l awski}
\email{grzes@math.uni.wroc.pl}

%\address{Instytut Matematyczny, Uniwersytet Wroc\l awski, Pl. Grunwaldzki 2/4, 50--384 Wroc\l aw, Poland}

\thanks{G.\  Plebanek was partially supported by NCN grant 2013/11/B/ST1/03596 (2014-2017).}

\begin{abstract}
Assuming $\be=\con$ (or some weaker statement), we construct a compactification $\gamma\omega$
  of $\omega$ such that its remainder $\gamma\omega\setminus\omega$ is nonseparable and carries
  a strictly positive measure.
\end{abstract}

\maketitle

\noindent{\tiny {\bf AMS subject classification: }54D35, 54D65}
\section{Introduction}

We consider here compactifications $\gamma\omega$ of the discrete space $\omega$.
Given a compact space $K$, we say that $K$ is a growth of $\omega$ if there is a compactification
  $\gamma\omega$ with the remainder $\gamma\omega\sm\omega$ homeomorphic to the space $K$.
It is well-known that every separable compactum is a growth of  $\omega$.

By a well-known theorem due to Parovi\u cenko  \cite{Pa63}, under the continuum hypothesis
  every compact space $K$ of topological weight $\le\con$ is an continuous image of the remainder
  $\beta\omega\sm\omega$ of the \v{C}ech-Stone compactification of $\omega$ and, consequently,
  $K$ is homeomorphic to the remainder $\gamma\omega\sm\omega$ of some compactification
  $\gamma\omega$.

Let $S$ be the Stone space of the measure algebra which may be seen as the quotient of
  $Bor(2^\omega)$ modulo the ideal of null sets.
Then $S$ is a nonseparable compact space that carries a strictly positive (regular probability
  Borel) measure, i.e.\ a measure that is positive on every nonempty open subset of $S$.
Since the topological weight of $S$ is $\con$, CH implies  that  $S$ is a growth of $\omega$.
In fact, this may be done in such a way  that the canonical measure on $S$ is defined by the
  asymptotic density defined for subsets of $\omega$, see Frankiewicz and Gutek \cite{FG81}.

Dow and Hart \cite{DH00}  proved  that the space $S$ is not a growth of $\omega$ if
  one assumes the Open Coloring Axiom (OCA).
Therefore it seems to be interesting to investigate whether one can always, in the usual set theory,
  construct a compactification $\gamma\omega$ such that the remainder $\gamma\omega\sm\omega$
  is nonseparable but carries a strictly positive measure.

If a compact space carries a strictly positive measure then it satisfies $ccc$;
  the converse does not hold which was already demonstrated by Gaifman \cite{Ga64}.
Later  Bell \cite{Be80}, van Mill \cite{Mi82}, \stevo\ \cite{To92} constructed several interesting
  examples of compactifications of $\omega$ having  nonseparable $ccc$ remainders; cf.\
 \stevo\  \cite{To00}.
It seems that the structure of all those examples exclude the possibility that $ccc$ could
  be strenghten to saying that the remainder in question supports  a measure,
  see e.g.\ Lemma 3.2 in D\v{z}amonja and Plebanek \cite{DP08}.

Before we state our main result we need to fix some notation and terminology concerning
  the set-theoretic assumption that we use.
We denote by $\lambda$ the usual product measure on the Cantor set $2^\omega$.
Let $\cE$ be an ideal of subsets $A$ of $2^\omega$ for which $\lambda(\overline{A})=0$.
Recall that the covering number $\cov(\cE)$ is the least cardinality of a covering of
  $2^\omega$ by sets from $\cE$; cf.\ Bartoszy\'nski and Shelah \cite{BS92} for cardinal
  invariants of the ideal $\cE$.
We shall sometimes write $\kappa_0=\cov(\cE)$ for simplicity.
As usual, $[\kappa_0]^{\le\omega}$ denotes the family of all countable subsets of $\kappa_0$.
Recall that $\cof [\kappa_0]^{\le\omega}$,  the cofinality of this partially ordered set,
  is the least size of a family $\cJ\sub [\kappa_0]^{\le \omega}$ such that every countable
  subset of $\kappa_0$ is contained in some $J\in\cJ$.
Our set-theoretic assumption (*) involves also $\be$, the familiar bounding number and
  reads as follows.
 \[ (*)\quad \mbox{writing }\kappa_0=\cov(\cE) \mbox{ we have } \cof [\kappa_0]^{\le\omega}\le\be.\]

\begin{thm}\label{main}
Assuming  (*) there is a compactification $\gamma\omega$ of the set of natural numbers such that its
  remainder $\gamma\omega\setminus\omega$ is not separable but carries a strictly positive
  regular probability Borel measure.
\end{thm}

Note that $(*)$ holds whenever $\be=\con$ or $\kappa_0=\omega_1$.
We do not know whether Theorem \ref{main} can be proved in the usual set theory.
In connection with the result of Dow and Hart mentioned above it is worth remarking that under OCA,
  $\be=\omega_2$ and we can further assume $\omega_2=\con$ (see Moore \cite{Mo02})
  so the compactification we construct here may exist even when the Stone space of the measure
  algebra is not a growth of $\omega$.

We remark that in our proof of Theorem \ref{main} we construct $\gamma\omega$
  such that $\gamma\omega\sm\omega$ supports a measure $\mu$ of countable Maharam type
  (meaning that $L_1(\mu)$ is separable).
We might, however, modify the construction so that the resulting $\mu$ will be of type $\kappa$.

The paper is organized as follows. In section 2 we formulate Theorem \ref{main} in terms of subalgebras of $P(\omega)$ and
  finitely additive measures defined on them, see Theorem \ref{2:2}.
Section 3 describes an inductive construction using (*) that leads to \ref{2:2}.
The key lemma showing that the inductions works is postponed to the final section 4.

\section{Compactifications and Boolean algebras}

We  denote by $\fin$ a family of finite subsets of $\omega$.
In the sequel, we shall consider Boolean algebras (of sets) $\alga$ such that
  $\fin\sub\alga\sub P(\omega)$.
Every such an algebra $\alga$ determines a compactification $K_\alga$ of $\omega$,
  where $K_\alga$ may be seen as the Stone space of all ultrafilters on $\alga$.
Then the algebra $\clop(K_\alga^*)$ of the clopen subsets of the remainder
  $K_\alga^*=K_\alga\setminus\omega$ is isomorphic to the quotient algebra $\alga/\fin$.
Hence $K_\alga^*$ is not separable if and only if $\alga/\fin$ is not $\sigma$-centred.

Given an algebra $\alga$ such that $\fin\sub\alga\sub P(\omega)$, we shall consider finitely
  additive probability measures $\mu$ on $\alga$ that vanish on finite sets.
Such a measure $\mu$  defines, via the Stone isomorphism,  a finitely additive measure
  $\widehat{\mu}$ on $\clop(K_\alga)$.
Then $\widehat{\mu}$ extends uniquely to a regular probability Borel measure $\overline{\mu}$ on
  $K_\alga$ such that $\overline{\mu}(K_\alga^*)=1$.
Note that the resulting Borel measure will be strictly positive whenever $\mu$ has the property
  mentioned in the following definition.

\begin{mydef}\label{2:1}
If $\fin\sub\alga\sub P(\omega)$ and $\mu$ is finitely additive measure on $\alga$
  then we shall say that {\em $\mu$ is almost strictly positive} if for every $A\in\alga$,
  $\mu(A)=0$ if and only if $A\in\fin$.
\end{mydef}

We can summarise our preliminary remarks and conclude that Theorem \ref{main} is an immediate
  consequence of the following result.

\begin{thm}\label{2:2}
Assume $(*)$.
There exists a Boolean algebra $\alga$ such that $\fin\subseteq\alga \subseteq P(\omega)$
  and a finitely additive probability measure $\mu$ on $\alga$ such that

  \begin{enumerate}[(a)]
    \item $\alga/\fin$ is not $\sigma$-centred;
    \item $\mu$ is almost strictly positive.
  \end{enumerate}
\end{thm}

We shall prove Theorem \ref{2:2} by inductive construction,
  gradually enlarging Boolean algebras and extending measures.
If $X\subseteq \omega$ then $\alga[X]$ stands for an algebra generated by $\alga \cup \{X\}$.
It is easy to check that $\alga[X]=\{ (A\cap X)\cup (A'\setminus X)\colon A,A'\in\alga\}$.

If $\mu$ is finitely additive on $\alga$ and $Z\sub\omega$ then we write
  \[\mu_*(Z)=\sup\{\mu(A)\colon A\in\alga, A\sub Z\}, \quad
    \mu^*(Z)=\inf\{\mu(A)\colon A\in\alga, A\supseteq Z\},\]
for the corresponding inner- and outer-measure.
Recall  the following well-known fact about extension of measures,
  see {\L}o{\'s} and Marczewski \cite{LM48}.

\begin{prop}\label{2:3}
The formula
\[ \widetilde\mu\big((A\cap X)\cup(A'\setminus X)\big) = \mu_*(A\cap X) + \mu^*(A'\sm X),\]
defines an extension of $\mu$ to a finitely additive measure $\widetilde\mu$ on $\alga[X]$.
\end{prop}

\section{A construction}

In this section we prove Theorem \ref{2:2}.
We start by fixing a countable dense subset $D$ of the Cantor space $2^\omega$ playing the role
  of  $\omega$ (and $\fin$ stands for finite subsets of $D$).
Let $\alga_0$ be the subalgebra of $P(D)$ generated by $\{C\cap D: C\in \clop(2^\omega)\}$
  and all finite subsets of $D$.
Further let $\mu_0$ be a measure on $\alga_0$ defined by
\[ \mu_0(C\btu F)=\lambda(C)
%  \marginpar{\tiny $\algc = \clop(2^\omega)$?}
 \mbox{ for } C\in\clop(2^\omega) \mbox{  and } F\in\fin.\]
Clearly, $\mu_0$ is well-defined and almost strictly positive on $\alga_0$.
Note that $\mu_0$ is nonatomic, in the sense that every element of $\alga_0$ may be written
  as a finite union of sets of arbitrary small measure.

Our construction requires a certain bookkeeping of families of sequences in $\alga_0$.
It will be convenient to use the following definition.

\begin{mydef}\label{3:0}
  A countable family $\cS\sub (\alga_0)^\omega $ will be  called an \emph{$s-$family}
    if  for every $S=(S(k))_k\in\cS $
  \begin{enumerate}[(i)]
    \item $S(0)\supseteq S(1)\supseteq\ldots$ and $\bigcap_k S(k)=\emptyset$;
    \item $\lim_k \lambda_0 (S(k))=0$.
  \end{enumerate}
\end{mydef}

Let $\cS$ be a countable infinite $s-$family together with  some fixed enumeration
  $\cS=\{S_n:n\in\omega\}$.
Then for any $\vf\in\omega^\omega$ we write
   \[  X_\varphi(\cS) = \bigcup_{n\in\omega} S_n\big(\varphi(n)\big) . \]

We shall write below  $\kappa_0=\cov(\cE)$ and $\kappa= \cof[\kappa_0]^{\le\omega}$ for simplicity.

Choose a family $\{Z_\alpha\colon\alpha<\kappa_0\}$ of closed subsets of $2^\omega$ such that
  $\lambda(Z_\alpha)=0$ for every $\alpha<\kappa_0$ and
  $\bigcup_{\alpha<\kappa_0}Z_\alpha=2^\omega$.
To every $Z_\alpha$ we associate $S_\alpha\in\alga_0^\omega$ as follows.
Fix  a bijection $d\colon\omega\to D$.
Write $Z_\alpha$ as an intersection of a decreasing sequence of $C_{\alpha, k}\in\clop(2^\omega)$
  and set \[S_\alpha(k)=C_{\alpha,k}\cap D\sm \{d(0) ,\ldots,d(k-1) \}.\]

Let $\cJ$ be a cofinal family in $[\kappa_0]^{\le\omega}$ of cardinality $\kappa$.
Given $J\in\cJ$, write $\cS^J=\{S_\alpha:\alpha\in J\}$.
Then $\cS^J$ is an $s-$family in our terminology.
For any $\vf\in\omega^J$ we put \[ X_\vf(\cS^J)=\bigcup_{\alpha\in J} S_\alpha(\vf(\alpha)).\]

\begin{lemma}\label{3:1}
Assume $(*)$.
There are a family $(\alga_\xi)_{\xi<\kappa}$ of Boolean subalgebras of $P(D)$
  and a family $(\mu_\xi)_{\xi<\kappa}$, where every $\mu_\xi$  is a finitely additive
  measure defined on $\alga_\xi$, such that,
  writing $\alga=\bigcup_{\xi<\kappa} \alga_\xi$, we have

  \begin{enumerate}[(i)]
    \item $|\alga_\xi|<\kappa$ for every $\xi<\kappa$;
    \item $\mu_\xi$ is almost strictly positive on $\alga_\xi$;
    \item $\mu_\eta|\alga_\xi=\mu_\xi$ whenever $\xi<\eta<\kappa$;
    \item for every $J\in\cJ $  there is $\vf\in\omega^J$ such that $ X_\vf(\cS^J)\in\alga$
      and $\omega\sm X_\vf(\cS^J)$ is infinite.
  \end{enumerate}
\end{lemma}

\begin{proof}
Enumerate $\cJ$ as $(J_\xi)_{1\le \xi<\kappa}$.
We start from $\alga_0$ and $\mu_0$ defined at the beginning of this section.
Fix $\xi<\kappa$.
Given $\alga_\beta$ and $\mu_\beta$ for $\beta<\xi$ we apply
  Lemma \ref{lemma:C} from section 4 to the algebra $\algb=\cup_{\beta<\xi}\alga_\beta$,
  the measure $\nu$ on $\algb$ which extends all $\mu_\beta$, $\beta<\xi$ and
  an $s-$family $\cS=\cS^{J_\xi}$.
Then Lemma \ref{lemma:C} gives us a suitable $X=X_\vf(\cS^{J_\xi})$ and
 we let $\alga_\xi$ be $\algb[X]$ and $\mu_\xi$ be an extension of $\nu$ to an
 almost strictly positive measure on $\alga_\xi$.
\end{proof}

Now we check that the algebra $\alga$ satisfying (i)-(iv) of Lemma \ref{3:1}
  is the one we are looking for.

\begin{lemma}\label{3:2}
If $\alga$ is an algebra is in \ref{3:1} then $\alga$ carries an almost strictly
  positive finitely additive probability measure and $\alga/\fin$ is not $\sigma$-centred.
\end{lemma}

\begin{proof}
It is clear that if we let $\mu$ be the unique measure on $\alga$ extending all
  $\mu_\xi $'s then $\mu$ is as required.

Checking that $\alga/\fin$ is not $\sigma$-centred amounts to verifying that
  if  $\{ p_k:k\in\omega\}$ is a family of nonprincipial ultrafilters on $\alga$ then
  there is an infinite $A\in\alga$ such that $A\notin p_k$ for every $k$.

Clearly every $p_k$ defines the unique $t_k\in 2^\omega$ which is in the intersection
  \[\bigcap \{C\in\clop(2^\omega): C\cap D\in p_k\}.\]
Then $t_k\in Z_{\alpha_k}$ for some $\alpha_k<\kappa_0$.
An $s-$family $\{S_{\alpha_k}: k\in\omega\}$ is contained in $\cS^J$ for some $J\in\cJ$.
Take $X_\vf(\cS^J)$ as in \ref{3:1}(iv).
Then $A=D\sm  X_\vf(\cS^J)$ is an infinite set lying outside every $p_k$, as required.
\end{proof}

\section{Key lemma}

We shall now  prove an auxiliary result, stated below as  Lemma \ref{lemma:C}, showing that the inductive construction of Lemma \ref{3:1} can  be carried out.
We follow here the notation introduced in section 3; in particular, the notion of an $s$-family was introduced in \ref{3:0}.
Recall also that, given an $s$-family $\cS=\{ S_i:i\in I\}$,  any $J\sub I$  and a function $\vf\in\omega^J$, we write $X_\vf(\cS)=\bigcup_{i\in J} S_i(\vf(i))$.
We sometimes write $X_\vf$ rather than $X_\vf(\cS)$ if $\cS$ is clear from the context.

\begin{lemma}\label{lemma:A}
  Let $\cS = \{S_0, \ldots, S_{n-1}\}$ be a fixed finite s-family. For a given infinite set $A\sub\omega$ we put
  \[
  \Phi_n^A = \{\varphi \in \omega^n \colon |A\setminus X_\varphi(\cS)| = \omega\}.\]
 If we consider $\Phi_n^A$ with the natural partial order then it has finitely many minimal elements.
\end{lemma}

\begin{proof}
  It is pretty obvious that for every $\vf\in \Phi_n^A$ there is minimal $\vf'\in\Phi_n^A$ such that $\vf'\leq \vf$.

  Fix an infinite set  $A\sub\omega$.
  We prove the lemma by induction on $n$.
  For $n=1$ it is trivial since $\Phi_1^A \subseteq \omega$  is well-ordered.
  Assume that for any $k < n$ and any infinite $B\sub\omega$ the family  $\Phi_k^B$ has finitely many minimal elements.

  Fix some minimal $\varphi_0 \in \Phi_n^A$.
  Consider any non-empty set $I\subseteq n$ of size less than $n$ and any function
  $\psi \in \omega^I$ such that $\psi \leq \varphi_0\!\upharpoonright\! I$
  and $A\setminus X_\psi$ is infinite
  (it is important that there is only finitely many such $\psi$ and sets $I\subseteq n$).
  By inductive assumption applied to  $A\sm X_\psi$ and an $s-$family indexed by  $n\sm I$,
  there exist finitely many minimal elements in $\Phi_{n\sm I}^{A\sm X_\psi}$ (call them $\varphi^j$).
  Observe that any minimal element $\varphi \in \Phi_n^A$ is of the form
  $\varphi = \psi \cup \varphi^j$ for some $j$, so the proof is complete.
\end{proof}

\begin{lemma}\label{lemma:B}
  Let $\algb \subseteq P(D)$ be a Boolean algebra of size less than $\be$ and let $\cS = \{S_n:n\in\omega \}$
  be a fixed $s-$family in $\algb^\omega$.

  There exists $g_0 \in \omega^\omega$ such that
  for any $g\in\omega^\omega$ with $g \geq g_0$,
  whenever $A\in\algb$ and $A\sub X_g$ then $A\sub\bigcup_{j\le N} S_j(g(j))$ for some $N\in\omega$.

 %  the following condition is satisfied:
  %\begin{equation}\label{eq:w_skonczenie_wielu}
   % \text{if } A \in \algb \text{ and } A\subseteq X_g
  %  \text{ then there exists $N$ such that }
   % A\subseteq \bigcup_j S_j\big(g(j)\big).
 % \end{equation}
\end{lemma}

\begin{proof}
  Fix $A\in\algb$. It follows from Lemma \ref{lemma:A} that for any $n\in\omega$ there exists a finite set
    $I_n \subseteq \omega$ such that if $\varphi \in \omega^n$ and
    $|A\sm X_\varphi|=\omega$ then $(A\sm X_\varphi) \cap I_n \neq \emptyset$.
  We inductively define a function $h_A\in\omega^\omega $ such that
 % \begin{equation}\label{eq:sn_cap_cup}
 \[S_n\big(h_A(n)\big) \cap \bigcup_{j\le n} I_j = \emptyset,\]
 % \end{equation}
  which can be done since   $(S_n(k))_{k}$  is a decreasing sequence with empty intersection.

  Now the family of functions $\{h_A:A\in\algb\}$ is of size less than $\be$  so there exists $g_0 \in \omega^\omega$
  such that $h_A \leq^* g_0$ for any $A\in \algb$. We shall check that $g_0$ is as required.

 % Now it is enough to check whether the function $g_0$ satisfies the condition
  %  \eqref{eq:w_skonczenie_wielu}:
    Take any $A\in\algb$ and $g\geq g_0$, and suppose that $ A\subseteq X_g$.
    Let $N\in\omega$ be such that $h_A(n) \le g(n)$ for any $n \ge N$. Write $\vf$ for the restriction of $g$ to $N$.
 Note that the set $A\sm X_\vf$ must be finite; indeed, otherwise $\vf\in \Phi_N^A$ so $J=(A\sm X_\vf)\cap  I_N\neq\emptyset$.
  But $I_N\cap S_n(g(n))=\emptyset$ for every $n\ge N$, hence $J\sub A\sm X_g$ which means that $A$ is not contained in $X_g$, contrary to our assumption.

  As  $A\sm X_\vf$ is finite, the lemma follows.
  \end{proof}

We are now ready for the key lemma; recall that the measure $\mu_0$ on $\alga_0$ was introduced in section 3.

\begin{lemma}\label{lemma:C}
Let $\algb \subseteq P(D)$ be a Boolean algebra of size  $<\be$  containing $\alga_0$
and let $\nu$ be a finitely additive almost strictly positive  probability measure on $\algb$ extending $\mu_0$.
Let $\cS=\{S_n:n\in\omega\} $ be a fixed  $s-$family contained in $\alga_0^\omega$.

There exists $g \in \omega^\omega$ such that for  $X_g=X_g(\cS)$,
  $\omega\setminus X_g$ is infinite and $\nu$ can be extended to an almost strictly positive
  measure on $\algb[X_g]$.
\end{lemma}

\begin{proof}
  For any $X \subseteq P(\omega)$ we may consider  an~extension $\widetilde\nu$ of $\nu$ to a
    finitely additive measure on  $\algb[X]$ given by the formula as in Proposition \ref{2:3}.
  The plan is  to find a function $g \in\omega^\omega$ such that the measure
    $\widetilde\nu$ defined in this way is almost strictly positive on $\algb[X_g]$.
  Observe that the extended measure $\widetilde\nu$ \emph{will} be almost strictly
    positive if for every $A\in\algb$

  \begin{enumerate}[\ref{lemma:C}(i)]
  \item  $\nu_*(A\cap X_g) > 0$ whenever  $A\cap X_g$ infinite,
  \item $\nu^*(A\setminus X_\vf) >0$ whenever  $A\sm X_\vf$ infinite,
\end{enumerate}
so we need to find $g$ for which (i) and (ii) are satisfied..

Fix infinite $A\in\algb$ and $n\in\omega$.  Consider $\Phi_n^A$ from
 Lemma \ref{lemma:A}; as this set contains finitely many minimal elements and $\nu$
  is almost strictly positive, we can choose
  $\varepsilon_n^A>0$ such that $\nu(A\setminus X_\varphi) \ge \varepsilon_n^A$
  for every $\varphi \in \Phi_n^A$. Note that the sequence of $\eps^A_n$ is decreasing.
    % Observe that $\varepsilon_n^A < \varepsilon_k^A$ for $n < k$.

  We inductively define  $h_A\in\omega^\omega$ so that
 % Take $f = g_0\in\omega^\omega$ from lemma \ref{lemma:B}.
  %If $A\setminus X_f$ is finite, define $h_A = f$.
 % Suppose it is not finite.
 % Define $h_A$ step by step: for each $n\in \omega$ and each $A\in\algb$ let $h_A(n)$
  %  be the smallest number such that
  %\begin{equation}
   \[ \nu\Big(S_n\big(h_A(n)\big)\Big) \leq \frac{\varepsilon_n^A}{2^{n+2}}.\]
  %\end{equation}
  This can be done since $S_n(k)\in\alga_0$, $\nu|\alga_0=\mu_0$ and
    $\lim_{k\to\infty} \mu_0\big(S_n(k)\big)=0$ by our definition of an $s-$family.

  Let $g_0\in\omega^\omega$ be a function as in Lemma \ref{lemma:B}
  Since $|\algb|  < \be$, there exists a function $g_1\in\omega^\omega$ such that $g_1(n)\ge g_0(n)$
    for every $n$ and $h_A\le^* g_1$ for every $A\in\algb$;
    say that $h_a(n)\le g_1(n)$ for every $n\ge N_A$.
  \medskip

   {\sc Claim.} For every $g\ge g_1$, if  $A\in\algb$ and  $A\setminus X_g$ is
    infinite then $\nu^*(A\setminus X_g) \ge \frac{1}{2} \varepsilon_{N_A}^A$.
  \medskip

  Let $\vf$ be the restriction of $g$ to $N_A$; write $Y=\bigcup_{n=N_A}^\infty S_n(g(n))$.
  With this notation we have $A\sm X_g=(A\sm X_\vf)\sm Y$, where $A\sm X_\vf\in \algb$.

  Note first that $\nu^*(A\sm X_g)\ge \nu(A\sm X_\vf)-\nu_*(Y)$;
    indeed if $B\in\algb$, $B\supseteq A\sm X_g$ then $(A\sm X_\vf)\sm B\sub Y$ so
    \[\nu(A\sm X_\vf)-\nu(B)\le \nu((A\sm X_\vf)\sm B)\le \nu_*(Y).\]
  Further note that $\nu(A\sm X_\vf)\ge \eps_{N_A}^A$ so to verify Claim it is sufficient
    to check that $\nu_*(Y)\le \eps_{N_A}^A/2$.
  Since $g_1\geq g_0$, it follows from Lemma \ref{lemma:B} that whenever $B\in\algb$
    and $B\sub Y$ then $B$ is contained in  $\bigcup_{N_A\le n\le k} S_n(g(n))$ for some $k$.
  Therefore, using finite additivity of $\nu$, we get
  \[
    \nu(B) \le
    \sum_{n=N_A }^k \nu\Big(S_n\big(g(n)\big)\Big) \le
    \sum_{n=N_A}^k \frac{1}{2^{n+2}} \varepsilon_n^A \le
  \frac{1}{2^{N_A+1}}\varepsilon_{N_A}^A  \le \frac{1}{2}\varepsilon_{N_A}^A .
  \]
  This shows that indeed $\nu_*(Y)\le \eps_{N_A}^A/2$, and the proof of Claim is complete.
\medskip

Note that by an analogous argument we check that if $g\ge g_1$ then $\nu_*(X_g)\le 1/2$
  so $\omega\sm X_g$ must be infinite.

 We have proved that every $g\geq g_1$ guarantees \ref{lemma:C}(ii) so to complete the proof
  we need to find $g\geq g_1$ for which \ref{lemma:C}(i) holds.

 We again apply a diagonal argument using $|\algb|<\be$.
 Given $A\in\algb$, let $f_A(n)=g_1(n)$ if $|A\cap S_n(k)|=\omega$ for every $k$.
 Otherwise, if $A\cap S_n(k)$ is finite for some $k$ we can take $f_A(n)\ge f_1(n)$
  such that $A\cap S_n(f_A(n))=\emptyset$ (recall that $\bigcap_k S_n(k)=\emptyset$).
 Finally, let $g\in\omega^\omega$ be a function that eventually dominates every $f_A$, $A\in\algb$.
 If $A\in\algb$ and $\nu_*(A\cap X_g)=0$ then $A\cap S_n(g(n))$ is finite for every $n$.
 Taking $N$ such that $g(n)\ge f_A(n)$ for every $n\ge N$, and writing $\vf$ for the restriction
  of $g$ to $N$ we get  $A\sub X_\vf$ so $A$ is finite itself.
 Now $g$ is so that both \ref{lemma:C}(i)-(ii) are satisfied, and the proof is complete.
 \end{proof}

\end{document}